\documentclass[12pt,reqno,twoside]{amsart}

\usepackage[T1]{fontenc}
\usepackage[utf8]{inputenc}
\usepackage{lmodern}
\usepackage{microtype}
\usepackage[a4paper,inner=1in,outer=1in,top=1.2in,bottom=1.2in]{geometry}

\usepackage{setspace}
\usepackage{amsmath,amssymb,amsthm,mathtools}
\usepackage{upgreek}
\usepackage{mathrsfs}
\usepackage{dsfont}

\usepackage{graphicx,tikz,caption,subcaption}
\usepackage{enumitem}

\usepackage[colorlinks=true,linkcolor=blue!50!black,citecolor=green!40!black,urlcolor=magenta!60!black]{hyperref}
\usepackage[nameinlink,capitalize]{cleveref}

\numberwithin{equation}{section}

\allowdisplaybreaks
\theoremstyle{plain}
\newtheorem{Th}{Theorem}[section]
\newtheorem{Lem}[Th]{Lemma}

\theoremstyle{definition}
\newtheorem{Def}[Th]{Definition}
\newtheorem{Rem}[Th]{Remark}

\DeclareMathOperator{\supp}{supp}

\begin{document}

\title[Semilinear Heat Inequalities in an Exterior Geodesic Domain on $\mathbb{S}^N$]
{Semilinear Heat Inequalities with a Hardy-Type Potential in an Exterior Geodesic Domain on $\mathbb{S}^N$}

\author[M. Jleli]{Mohamed Jleli}
\address{(M. Jleli) Department of Mathematics, College of Science, King Saud University, Riyadh 11451, Saudi Arabia}
\email{jleli@ksu.edu.sa}

\author[B. Samet]{Bessem Samet}
\address{(B. Samet) Department of Mathematics, College of Science, King Saud University, Riyadh 11451, Saudi Arabia}
\email{bsamet@ksu.edu.sa}


\renewcommand{\subjclassname}{\textbf{2020 Mathematics Subject Classification}}

\begin{abstract}
We study an inhomogeneous semilinear heat inequality on the unit sphere \(\mathbb S^N\), \(N\ge3\), in an exterior geodesic domain associated with a fixed pole. The equation involves the singular Hardy-type potential \(\lambda/\sin^2 r\), where \(r=d(o,x)\), and the weighted nonlinearity \((\sin r)^\alpha |u|^p\). For \(\alpha>-2\) and \(0<\lambda\le \lambda^*=((N-2)/2)^2\), we prove the existence of a critical exponent \(p_{\mathrm{crit}}=p_{\mathrm{crit}}(\alpha,N,\lambda)\) governing the existence and nonexistence of solutions. More precisely, we prove that no weak solution exists for any nontrivial nonnegative source in the range \(p>p_{\mathrm{crit}}\), whereas classical solutions exist for some positive continuous sources in the range \(1<p<p_{\mathrm{crit}}\). Under suitable additional assumptions, we also prove nonexistence at the critical exponent \(p=p_{\mathrm{crit}}\). If \(\alpha\le -2\), we show that nonexistence holds for all \(p>1\). The analysis is based on the construction of radial Hardy barriers adapted to the antipodal singularity and on sharp integral estimates involving power and logarithmic cutoffs near \(r=\pi\).
\end{abstract}

\maketitle

\tableofcontents

\section{Introduction and main results}

For \(N\geq 3\), let \(\mathbb S^N\) denote the unit sphere endowed with its standard Riemannian metric, and let \(d\) be the associated geodesic distance. Fix a point \(o\in\mathbb S^N\), and let \(\delta\in(0,\pi)\). We consider the exterior geodesic domain
\[
\Omega_\delta=\{x\in\mathbb S^N:\delta<d(o,x)<\pi\},
\]
together with its partial closure
\[
\overline{\Omega_\delta}^{\,\pi}
=\{x\in\mathbb S^N:\delta\le d(o,x)<\pi\}.
\]
We also denote by
\[
\Gamma_\delta=\{x\in\mathbb S^N: d(o,x)=\delta\}
\]
the regular boundary of \(\Omega_\delta\); see Figure~\ref{fig:exterior-geodesic-domain}.

\begin{figure}[ht!]
\centering
\includegraphics[width=0.5\textwidth]{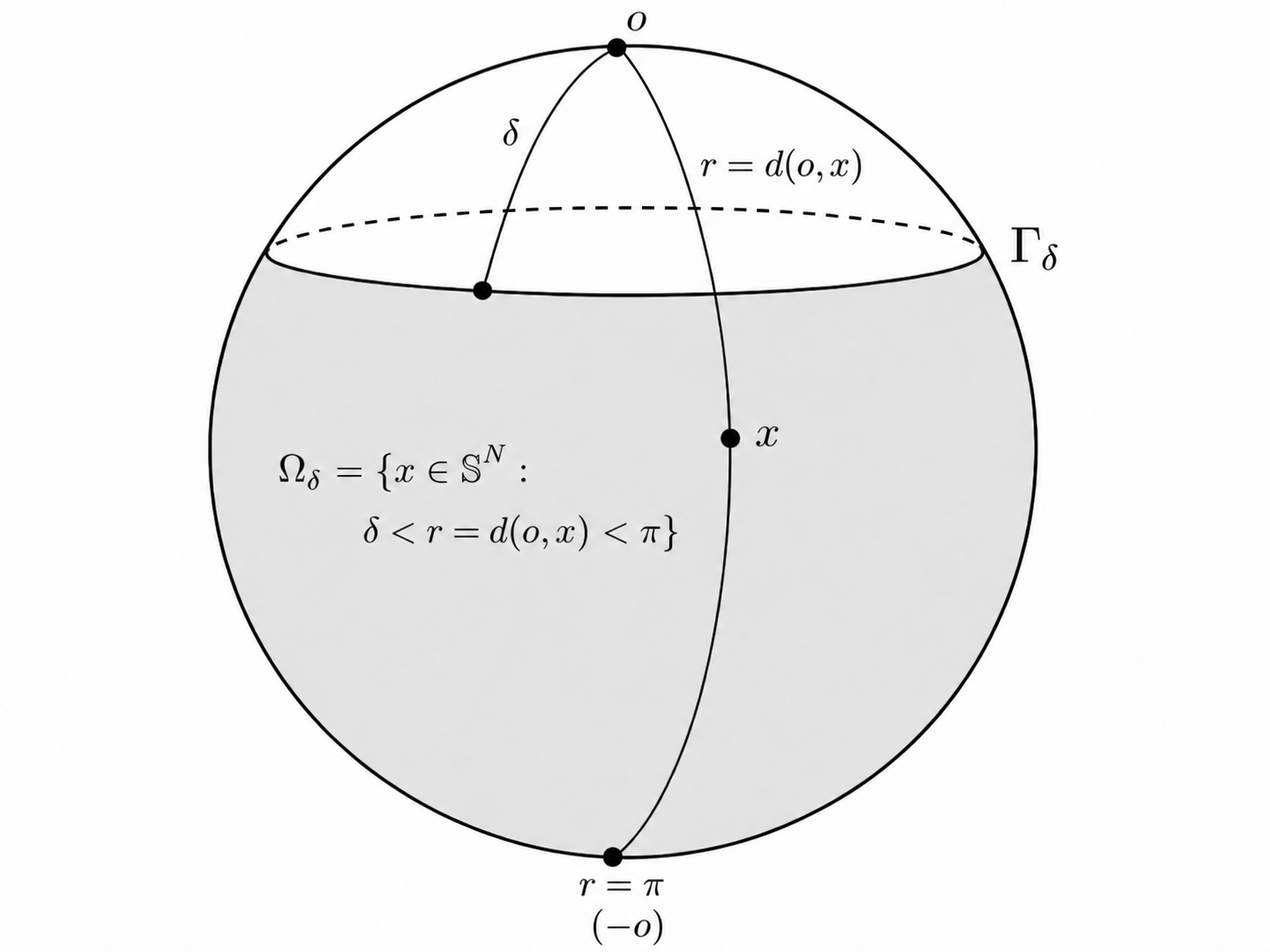}
\caption{The exterior geodesic domain \(\Omega_\delta\) on \(\mathbb S^N\).}
\label{fig:exterior-geodesic-domain}
\end{figure}

We are concerned with the critical behavior governing the existence and nonexistence of weak solutions to the inhomogeneous semilinear heat inequality
\begin{equation}\label{P}
\begin{cases}
\partial_t u-\Delta_{\mathbb{S}^N}u-\dfrac{\lambda}{\sin^2 d(o,x)}\,u
\ge
(\sin d(o,x))^{\alpha}|u|^p+f(x),
& \text{in } (0,\infty)\times\Omega_\delta,\\[6pt]
u\ge 0,
& \text{on } (0,\infty)\times\Gamma_\delta,
\end{cases}
\end{equation}
where \(u=u(t,x)\), \(\Delta_{\mathbb{S}^N}\) denotes the Laplace--Beltrami operator on \(\mathbb{S}^N\), \(p>1\), \(\alpha\in\mathbb{R}\), and
\[
f\in L^1_{\mathrm{loc}}(\overline{\Omega_\delta}^{\,\pi}), 
\qquad f\ge 0,\qquad f\not\equiv 0.
\]
We assume that
\begin{equation}\label{cd-lambda}
0<\lambda\le \lambda^*
=\left(\frac{N-2}{2}\right)^2.
\end{equation}
The potential \(\lambda/\sin^2 r\), with \(r=d(o,x)\), has a Hardy-type singularity at the antipodal point \(r=\pi\), since \(\sin r\sim \pi-r\) as \(r\to\pi^-\), and \(\pi-r=d(x,-o)\) is the geodesic distance from \(x\) to the antipodal point \(-o\).

Weak solutions to \eqref{P} are understood in the following sense. Let
\[
Q=(0,\infty)\times \overline{\Omega_\delta}^{\,\pi},
\qquad
\Sigma=(0,\infty)\times \Gamma_\delta.
\]
We define the class of admissible test functions by
\[
\mathcal{X}
=
\left\{
\xi \in C_c^{1,2}(Q):
\xi \ge 0,\ 
\xi =0 \ \text{on } \Sigma,\ 
\partial_\nu \xi \le 0 \ \text{on } \Sigma
\right\}.
\]
Here, \(C_c^{1,2}(Q)\) denotes the space of functions \(\xi\) that are compactly supported in \(Q\), of class \(C^1\) with respect to \(t\), and of class \(C^2\) with respect to \(x\). Moreover,
\[
\partial_\nu \xi
=
\left\langle \nabla_{\mathbb{S}^N}\xi,\nu\right\rangle,
\]
where \(\langle\cdot,\cdot\rangle\) denotes the Riemannian inner product on \(\mathbb{S}^N\), and \(\nu\) is the outward unit normal to \(\Gamma_\delta\) with respect to \(\Omega_\delta\).

\begin{Def}\label{def-ws}
We say that \(u\) is a weak solution to \eqref{P} if
\[
u \in L^p_{\mathrm{loc}}(Q)
\]
and, for every \(\xi \in \mathcal{X}\), one has
\[
\begin{aligned}
&\int_{Q} (\sin d(o,x))^{\alpha} |u|^p \, \xi \, d\mu \, dt
+ \int_{Q} f(x) \, \xi \, d\mu \, dt\\
&\le
- \int_{Q} u \, \partial_t \xi \, d\mu \, dt
- \int_{Q} u
\left(
\Delta_{\mathbb{S}^N} \xi
+
\frac{\lambda}{\sin^2 d(o,x)}\,\xi
\right)
d\mu \, dt.
\end{aligned}
\]
Here, \(d\mu\) denotes the Riemannian measure associated with the standard metric on \(\mathbb{S}^N\).
\end{Def}

The weak formulation in Definition~\ref{def-ws} is obtained by multiplying \eqref{P} by a test function \(\xi\in\mathcal{X}\) and integrating over \(Q\). Using the Green formula recalled in Section~\ref{sec2}, the boundary term involving \(\xi\,\partial_\nu u\) vanishes, since \(\xi=0\) on \(\Sigma\). The remaining boundary term is
\[
\int_{\Sigma} u\,\partial_\nu \xi \, d\sigma\,dt \le 0,
\]
because \(u\ge0\) and \(\partial_\nu\xi\le0\) on \(\Sigma\). Dropping this nonpositive term from the right-hand side yields the inequality stated in Definition~\ref{def-ws}.

The study of critical exponents for semilinear parabolic equations goes back to the classical work of Fujita \cite{Fujita}. For the Cauchy problem
\[
\begin{cases}
\partial_t u-\Delta u=u^p,
& \text{in }  (0,\infty)\times\mathbb{R}^N,\\[4pt]
u(0,x)=u_0(x)\ge0,
& \text{in } \mathbb{R}^N,
\end{cases}
\]
Fujita, see also \cite{Hayakawa}, identified the threshold
\[
p_F=1+\frac{2}{N}.
\]
More precisely, this threshold separates the nonexistence of global nontrivial nonnegative solutions for \(1<p\le p_F\) from the existence of global positive solutions for sufficiently small initial data when \(p>p_F\).

Many generalizations of Fujita's result have since been obtained; see, for instance, the survey paper by Deng and Levine \cite{DengLevine} and the references therein. In particular, Zhang \cite{Zhang} considered the inhomogeneous parabolic problem
\[
\begin{cases}
\partial_t u-\Delta u=u^p+f(x),\qquad u\geq 0,
& \text{in } (0,\infty)\times\mathbb{R}^N,\\[4pt]
u(0,x)=u_0(x),
& \text{in } \mathbb{R}^N,
\end{cases}
\]
where \(N\geq 3\) and \(f\ge0\) is nontrivial. He proved that, for this problem, the critical exponent is
\[
p^*=1+\frac{2}{N-2}.
\]
More precisely, if \(1<p\leq p^*\), then no global positive solution exists for any nontrivial nonnegative source \(f\). On the other hand, if \(p>p^*\), then global positive solutions exist for some nontrivial nonnegative \(f\) and suitable \(u_0\).

Parabolic equations involving Hardy-type potentials have attracted considerable attention. In \cite{BarasGoldstein}, Baras and Goldstein considered the problem
\[
\begin{cases}
u_t-\Delta u-\dfrac{\lambda}{|x|^2}u=f,
& \text{in } (0,\infty)\times \Omega,\\[6pt]
u(t,x)=0,
& \text{on } (0,\infty)\times \partial\Omega,\\[6pt]
u(0,x)=u_0(x),
& \text{in } \Omega,
\end{cases}
\]
where \(\Omega\subset \mathbb{R}^N\), \(N\ge 3\), is a bounded domain containing the origin. They showed that the existence and nonexistence of positive solutions depend critically on the value of the parameter \(\lambda\). More precisely, they proved that, for nonnegative \(L^2\) initial data and nonnegative source terms, the problem admits a unique global nonnegative weak solution when
\[
\lambda\le \lambda^*=\left(\frac{N-2}{2}\right)^2,
\]
whereas no solution exists, even locally in time, when \(\lambda>\lambda^*\). For further results related to this problem, see, for instance, \cite{Cabre,Goldstein,Vancostenoble}.

The corresponding semilinear inhomogeneous problem with nonlinearity \(u^p\), where \(u_0,f\ge0\) belong to suitable function classes, was later studied by Abdellaoui et al.~\cite{Abdellaoui}. They showed that, when \(\lambda>0\), there exists a critical exponent \(p_+(\lambda)\) such that, for \(p\ge p_+(\lambda)\), no solution exists for any nontrivial nonnegative initial datum. In \cite{Abdellaoui2}, the influence of Hardy-type potentials was investigated in the presence of gradient nonlinearities. Further results related to evolution equations and inequalities with Hardy-type potentials can be found in \cite{Abdellaoui3,Abdellaoui4,ArdilaHamanoIkeda,HAM,JSS,JSS2}.

Fujita-type phenomena have also been studied in non-Euclidean settings, including Riemannian manifolds. In this direction, Zhang \cite{Zhang99} extended several classical blow-up results for semilinear parabolic problems from \(\mathbb{R}^N\) to noncompact complete Riemannian manifolds under a volume growth assumption on geodesic balls. We also refer to Mastrolia et al.~\cite{MAS}, who investigated the nonexistence of nonnegative nontrivial weak solutions for a class of parabolic differential inequalities.

More recently, Gu et al.~\cite{GU} studied the semilinear parabolic problem
\[
\partial_t u-\Delta u+Vu=Wu^p,
\qquad \text{in } (0,\infty)\times M,
\]
where \(M\) is a connected noncompact geodesically complete Riemannian manifold. Their results provide criteria for the existence and nonexistence of global positive solutions in terms of the potential \(V\), the weight \(W\), and the geometry of \(M\). A key feature of their approach is the use of a positive solution \(h\) of
\[
\Delta h=Vh,
\]
which allows the equation to be transformed into a weighted heat equation through a Doob \(h\)-transform. The proof also relies on a suitable family of test functions adapted to the geometry of the manifold.

Further results on elliptic problems on Riemannian manifolds can be found in \cite{Grigoryan-Sun,Grigoryan-Sun2,JRS,MAS15} and the references therein.

The present work is motivated by the role of Hardy-type potentials in parabolic problems and by the work of Gu et al.~\cite{GU}, where the critical behavior is governed by the interaction between the potential, the nonlinear weight, and the geometry of the underlying Riemannian manifold.

We point out that the techniques used in the present paper differ from those in \cite{GU}, mainly because of the different geometric nature of the problem. Indeed, in \cite{GU}, the equation is studied on a connected noncompact geodesically complete Riemannian manifold without boundary, and the analysis relies on the geometry at infinity. In contrast, here the ambient manifold is the compact sphere \(\mathbb S^N\), while the problem is posed in the exterior geodesic domain \(\Omega_\delta\), with a boundary condition on \(\Gamma_\delta\) and a Hardy-type singularity at the antipodal point \(r=\pi\). Thus, the critical behavior is not determined by volume growth at infinity, but by the local spherical geometry near \(r=\pi\) and by the radial Hardy barrier associated with
\[
\Delta_{\mathbb S^N}+\frac{\lambda}{\sin^2 r},\qquad r=d(o,x).
\]
This barrier is chosen to vanish on \(\Gamma_\delta\), so that the resulting test functions are compatible with the boundary condition in the weak formulation.

Another difference is that the proof in \cite{GU} uses, after the \(h\)-transform, test functions involving negative powers of the transformed positive solution \(v=u/h\), namely functions of the form \(v^{-a}\varphi^b\). This relies on the positivity of the solution. In the present paper, no sign condition is imposed on \(u\) in the interior of the domain, and the nonlinearity is \(|u|^p\). Therefore, our test functions are independent of \(u\) and are instead built from the radial Hardy barrier and from cutoffs adapted to the boundary and to the antipodal singularity.

For \(N\ge 3\) and \(\lambda\) satisfying \eqref{cd-lambda}, define
\[
\lambda_N
=
\frac{N-2}{2}
-
\sqrt{
\left(\frac{N-2}{2}\right)^2-\lambda
}.
\]
Our main result is stated as follows.

\begin{Th}\label{T1.2}
Assume that \(N\ge 3\), \(\lambda\) satisfies \eqref{cd-lambda}, and \(\alpha\in\mathbb{R}\).
\begin{itemize}
\item[\rm(i)] There exists \(a_\lambda \in [0,\pi)\) such that, for every \(\delta\in (a_\lambda,\pi)\), the following nonexistence results hold.
\begin{itemize}
\item[\rm(a)] If
\begin{equation}\label{cd-nnex1}
p> \max\left\{1, 1+\frac{\alpha+2}{\lambda_N}\right\},
\end{equation}
then problem \eqref{P} admits no weak solution for any function
\[
f\in L^1_{\mathrm{loc}}(\overline{\Omega_\delta}^{\,\pi}),
\qquad
f\ge 0,\qquad f\not\equiv 0.
\]

\item[\rm(b)] Assume that
\begin{equation}\label{cd-lambda-alpha}
\alpha>-2,\quad 0<\lambda<\lambda^*
\qquad \text{or}\qquad
-2<\alpha<\frac{N-6}{2},\quad \lambda=\lambda^*.
\end{equation}
If
\begin{equation}\label{cd-nnex1-crit}
p=1+\frac{\alpha+2}{\lambda_N},
\end{equation}
then the same nonexistence conclusion holds.
\end{itemize}

\item[\rm(ii)] Assume that
\begin{equation}\label{cd-ex1}
\alpha>-2,\qquad 1<p< 1+\frac{\alpha+2}{\lambda_N}.
\end{equation}
Then there exists \(b\in [0,\pi)\) such that, for every \(\delta\in (b,\pi)\), problem \eqref{P} admits classical solutions
\[
u\in C^\infty((0,\infty)\times\Omega_\delta)\cap C(Q)
\]
for some  function
\[
f\in C(\overline{\Omega_\delta}^{\,\pi}),\qquad f>0.
\]
\end{itemize}
\end{Th}

\begin{Rem}
When \(\alpha\le -2\), only a nonexistence regime occurs for \(p>1\). Indeed,
\[
1+\frac{\alpha+2}{\lambda_N}\le 1,
\]
and therefore \eqref{cd-nnex1} is equivalent to \(p>1\). Thus, for every \(p>1\), problem \eqref{P} admits no weak solution for any nontrivial nonnegative source
\(f\in L^1_{\mathrm{loc}}(\overline{\Omega_\delta}^{\,\pi})\).
\end{Rem}

\begin{Rem}
Assume that \(\alpha>-2\). Then problem \eqref{P} admits a critical behavior governed by
\[
p_{\mathrm{crit}}=1+\frac{\alpha+2}{\lambda_N}.
\]
More precisely, for \(\delta\) sufficiently close to \(\pi\), Theorem~\ref{T1.2} gives nonexistence for any nontrivial nonnegative source when \(p>p_{\mathrm{crit}}\), and existence for some positive continuous source when \(1<p<p_{\mathrm{crit}}\). Moreover, under \eqref{cd-lambda-alpha}, the nonexistence result extends to the critical exponent \(p=p_{\mathrm{crit}}\). In particular, when \(\lambda=\lambda^*\), we have
\[
p_{\mathrm{crit}}=1+\frac{2(\alpha+2)}{N-2}.
\]
\end{Rem}

\begin{Rem}
For \(\alpha>-2\), Theorem~\ref{T1.2} gives a complete critical picture, except in the critical Hardy case (see Figure~\ref{fig:critical-regions})
\[
\lambda=\lambda^*,
\qquad
\alpha\ge \frac{N-6}{2},
\qquad
p=p_{\mathrm{crit}}.
\]

\begin{figure}[ht!]
\centering
\includegraphics[width=0.75\textwidth]{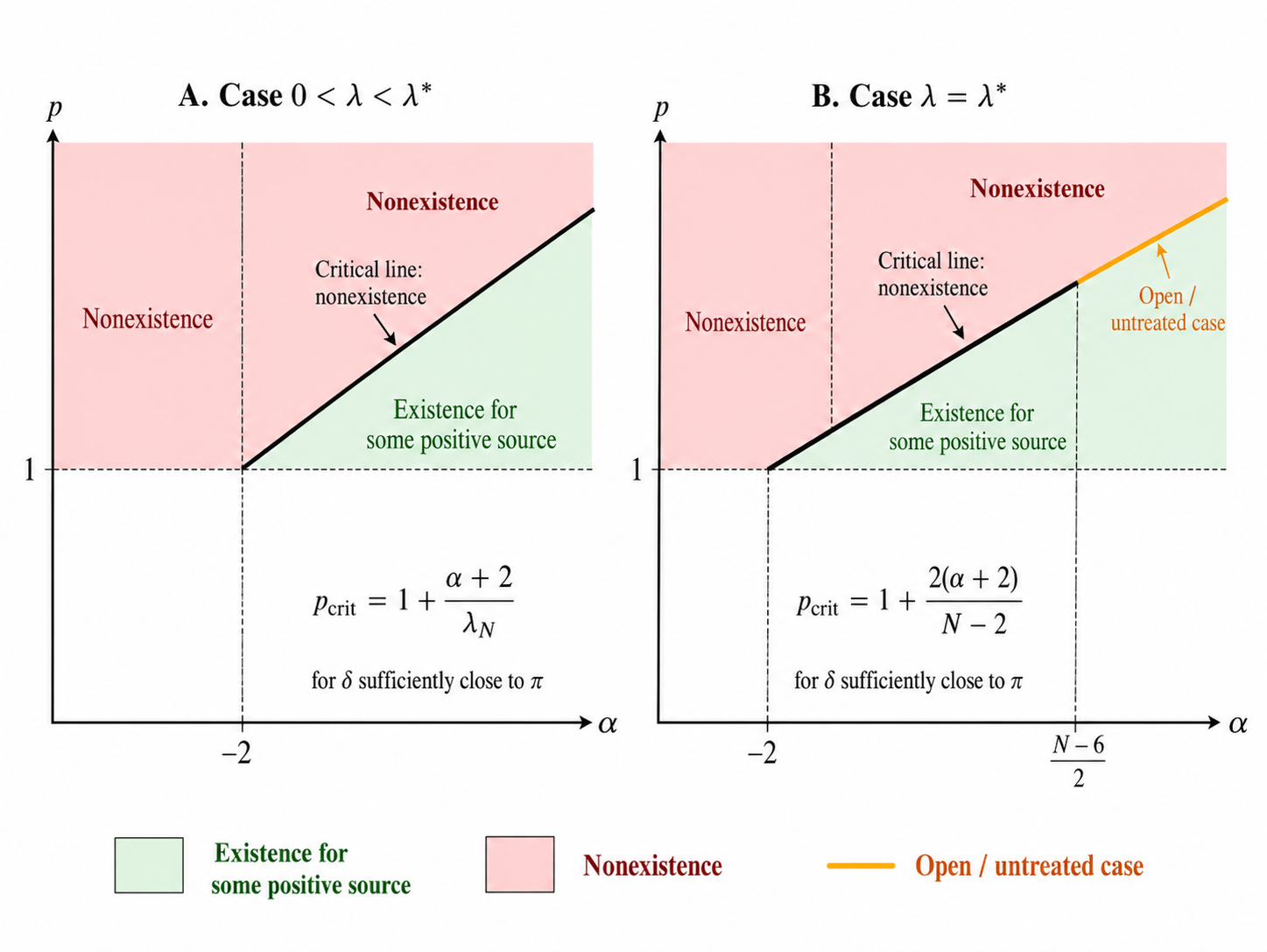}
\caption{Existence and nonexistence regions for problem \eqref{P}.}
\label{fig:critical-regions}
\end{figure}

\end{Rem}

Let us add a few comments on our approach. The nonexistence part is based on a suitable choice of test functions \(\xi\in\mathcal{X}\), adapted simultaneously to the singular Hardy potential, the geometry of the exterior geodesic domain \(\Omega_\delta\), the boundary condition on \(\Gamma_\delta\), and the required integral estimates. The main point is that the test functions must vanish in a small neighborhood of the singular point \(r=\pi\), satisfy the admissibility conditions on \(\Gamma_\delta\), and still retain enough mass on fixed subannuli in order to detect the nontrivial source term. To achieve this, we construct a positive radial Hardy barrier \(H\) satisfying
\[
\Delta_{\mathbb S^N}H+\frac{\lambda}{\sin^2 r}H=0,\qquad r=d(o,x),
\]
in a neighborhood of \(r=\pi\), together with
\[
H=0,\qquad \partial_\nu H<0
\quad \text{on } \Gamma_\delta.
\]
This barrier carries the precise Hardy singular profile and is used as the main spatial weight in the test functions.

For the supercritical nonexistence range, the spatial cutoff is chosen in the form
\[
H(x)\psi^m(R(\pi-r)),
\]
which removes a small neighborhood of the singular point \(r=\pi\) and localizes the estimates on a thin transition layer near that point. This produces the power-type quantity
\[
R^{\frac{\alpha+2}{p-1}-\lambda_N}.
\]
The condition that this power be negative is precisely
\[
p>1+\frac{\alpha+2}{\lambda_N},
\]
which is the supercritical nonexistence condition.

In the critical case \(p=p_{\mathrm{crit}}\), the power
\[
\frac{\alpha+2}{p-1}-\lambda_N
\]
vanishes, so the previous choice of cutoff is no longer sufficient to obtain a vanishing upper bound as \(R\to\infty\). To recover a decay factor, we use a logarithmic cutoff near the singular point \(r=\pi\), in the form
\[
H(x)\eta^m\left(
\frac{\ln(R(\pi-r))}{\ln\sqrt R}
\right).
\]
This choice localizes the estimates on the transition region
\[
\frac1R<\pi-r<\frac1{\sqrt R},
\]
and produces an additional factor of order \((\ln R)^{-1}\) in the derivative estimates. This gain is enough to treat the critical exponent when \(0<\lambda<\lambda^*\). In the critical Hardy case \(\lambda=\lambda^*\), the Hardy barrier itself carries an additional logarithmic behavior, which creates a logarithmic loss in the integral estimates and explains the extra restriction appearing in \eqref{cd-lambda-alpha} (see Figure~\ref{fig:critical-regions}).

For the existence part, we construct explicit positive stationary supersolutions and then turn them into time-dependent classical solutions of the inequality by multiplying them by a simple increasing time factor. The construction differs in the two cases \(0<\lambda<\lambda^*\) and \(\lambda=\lambda^*\), because the singular profiles of the associated Hardy operator are different.

The rest of the paper is organized as follows. In Section~\ref{sec2}, we recall some geometric preliminaries on \(\mathbb S^N\), including stereographic coordinates, geodesic polar coordinates, the expression of the Laplace--Beltrami operator, and the Green formula used in the weak formulation. In Section~\ref{sec3}, we develop the auxiliary tools needed for the proof. We first construct a positive radial Hardy barrier near the singular point \(r=\pi\). We then build admissible test functions based on this barrier, including a logarithmic cutoff for the critical exponent. We also prove the integral estimates needed for the proof of the main result. In Section~\ref{sec4}, we prove Theorem~\ref{T1.2}.

\section{Geometric preliminaries on \(\mathbb{S}^N\)}\label{sec2}

In this section, we recall some geometric facts on \(\mathbb{S}^N\) that will be used throughout the paper. For further details, we refer to \cite{Grigoryan}.

We denote by
\[
\mathbb{S}^N=\{y\in\mathbb{R}^{N+1}:|y|=1\}
\]
the unit sphere. For any point \(y\in\mathbb{S}^N\) different from the north pole, we denote its stereographic coordinate from the north pole by \(x\in\mathbb{R}^N\). In these coordinates, the Riemannian metric on \(\mathbb{S}^N\) is given by
\[
g_{ij}(x)=\frac{4}{(1+|x|^2)^2}\,\delta_{ij},
\qquad i,j=1,\dots,N.
\]
We set
\[
g^{ij}=(g_{ij})^{-1},
\qquad
g=\det(g_{ij}),
\qquad i,j=1,\dots,N.
\]

The contravariant components of the gradient of a smooth function \(u\) on \(\mathbb{S}^N\) are
\[
(\nabla_{\mathbb{S}^N}u)^i
=
\sum_{j=1}^N g^{ij}\frac{\partial u}{\partial x_j},
\qquad i=1,\dots,N.
\]
The Laplace--Beltrami operator on \(\mathbb{S}^N\) is given by
\[
\Delta_{\mathbb{S}^N} u
=
\frac{1}{\sqrt{g}}
\sum_{i=1}^N
\frac{\partial}{\partial x_i}
\left(
\sqrt{g}
\sum_{j=1}^N g^{ij}\frac{\partial u}{\partial x_j}
\right).
\]

For every open set \(\mathcal{O}\subset \mathbb{S}^N\) with \(C^1\) boundary, and for all
\(u,v\in C^2(\mathcal{O})\cap C^1(\overline{\mathcal{O}})\), the following Green formula holds:
\[
\int_{\mathcal{O}} v\,\Delta_{\mathbb{S}^N}u\,d\mu
=
\int_{\partial\mathcal{O}} v\,\partial_\nu u\,dS
-
\int_{\mathcal{O}}
\left\langle \nabla_{\mathbb{S}^N}u,\nabla_{\mathbb{S}^N}v\right\rangle
\,d\mu,
\]
where
\[
\partial_\nu u
=
\left\langle \nabla_{\mathbb{S}^N}u,\nu\right\rangle.
\]
Here, \(\nu\) is the outward unit normal to \(\partial\mathcal{O}\), \(\langle\cdot,\cdot\rangle\) is the Riemannian inner product on \(\mathbb{S}^N\), \(d\mu\) denotes the Riemannian measure on \(\mathbb{S}^N\), and \(dS\) denotes the induced Riemannian measure on \(\partial\mathcal{O}\).

Let \((r,\theta)\) be the geodesic polar coordinates centered at the fixed point
\(o\in\mathbb{S}^N\). For
\(x\in \mathbb{S}^N\setminus\{o,-o\}\), we write
\[
x=(r,\theta),
\qquad
r=d(o,x)\in(0,\pi),\quad \theta\in \mathbb{S}^{N-1}.
\]
In these coordinates, the standard metric on \(\mathbb{S}^N\) is given by
\[
ds^2=dr^2+(\sin r)^2 d\theta^2,
\]
where \(d\theta^2\) denotes the standard metric on \(\mathbb{S}^{N-1}\). Consequently,
\[
d\mu=(\sin r)^{N-1}\,dr\,d\sigma(\theta),
\]
where \(d\sigma(\theta)\) denotes the Riemannian measure on \(\mathbb{S}^{N-1}\).

In particular, if \(u=u(r)\) is radial, then
\[
\nabla_{\mathbb{S}^N}u=u'(r)\,\partial_r,
\qquad
|\nabla_{\mathbb{S}^N}u|=|u'(r)|,
\]
and
\[
\Delta_{\mathbb{S}^N}u
=
u''(r)+(N-1)\cot r\,u'(r).
\]

\section{Auxiliary results}\label{sec3}

This section contains the auxiliary tools needed in the proof of the main result. Throughout this section, we assume that
\[
N\ge 3,\qquad p>1,\qquad \alpha\in\mathbb{R},\qquad \delta\in(0,\pi),
\]
and that \(\lambda\) satisfies \eqref{cd-lambda}.

\subsection{Construction of a radial Hardy barrier in \(\Omega_\delta\)}

We consider the Hardy-type operator
\[
L_\lambda
=
\Delta_{\mathbb{S}^N}+\frac{\lambda}{\sin^2 r},
\qquad r=d(o,x).
\]
In this subsection, when \(\delta\) is sufficiently close to \(\pi\), we construct a radial function \(H(x)=h(r)\) satisfying
\begin{equation}\label{H-OK}
\begin{cases}
L_\lambda H=0, & \text{in } \Omega_\delta,\\[4pt]
H>0, & \text{in } \Omega_\delta,\\[4pt]
H=0, & \text{on } \Gamma_\delta,\\[4pt]
\partial_\nu H<0, & \text{on } \Gamma_\delta,
\end{cases}
\end{equation}
and study its behavior near the singular point \(r=\pi\).

Since \(\nu=-\partial_r\) on \(\Gamma_\delta\) and \(H\) is radial, \eqref{H-OK} reduces to
\begin{equation}\label{H-pb}
\begin{cases}
h''(r)+(N-1)\cot r\,h'(r)+\dfrac{\lambda}{\sin^2 r}\,h(r)=0,
& \delta<r<\pi,\\[6pt]
h(r)>0,
& \delta<r<\pi,\\[6pt]
h(\delta)=0,\\[4pt]
h'(\delta)>0.
\end{cases}
\end{equation}

The endpoint \(r=\pi\) is a regular singular endpoint. Its indicial equation is
\[
\mu(\mu+N-2)+\lambda=0.
\]
The corresponding roots are
\[
\mu_\pm
=
-\frac{N-2}{2}
\pm
\sqrt{\left(\frac{N-2}{2}\right)^2-\lambda}.
\]
We distinguish two cases.

\noindent\textbf{Case 1.} \(0<\lambda<\lambda^*\). In this case, we have
\[
\mu_-<-\frac{N-2}{2}<\mu_+<0.
\]
By the Frobenius theorem, there exists a solution \(h_\pi\) of the radial equation
\[
h_\pi''(r)+(N-1)\cot r\,h_\pi'(r)
+\frac{\lambda}{\sin^2 r}h_\pi(r)=0
\]
which is characterized near \(r=\pi\) by
\[
h_\pi(r)
=
(\pi-r)^{\mu_+}
\sum_{k=0}^{\infty}A_k(\pi-r)^k,
\qquad A_0>0.
\]
Since \(A_0>0\), it follows that
\[
h_\pi(r)>0
\]
for \(r\) sufficiently close to \(\pi\). We define
\begin{equation}\label{alambda-def}
a_\lambda=
\begin{cases}
\sup\{a\in(0,\pi):h_\pi(a)=0\}, & \text{if }h_\pi \text{ has zeros in }(0,\pi),\\[4pt]
0, & \text{otherwise}.
\end{cases}
\end{equation}
Then \(a_\lambda<\pi\), and
\[
h_\pi(r)>0
\qquad \text{for } a_\lambda<r<\pi.
\]

Let now \(\delta\in(a_\lambda,\pi)\). We define
\begin{equation}\label{h-def}
h(r)
=
h_\pi(r)
\int_\delta^r
\frac{ds}{(\sin s)^{N-1}h_\pi(s)^2},
\qquad \delta\le r<\pi.
\end{equation}
By the reduction-of-order formula, \(h\) is another solution of the same radial equation. Since \(h_\pi(r)>0\) for \(a_\lambda<r<\pi\), the integral in \eqref{h-def} is well defined and positive for every \(r\in(\delta,\pi)\). Hence,
\[
h(r)>0,
\qquad \delta<r<\pi.
\]
Moreover,
\[
h(\delta)=0,
\]
and
\[
h'(\delta)
=
\frac{1}{(\sin \delta)^{N-1}h_\pi(\delta)}
>0.
\]
Therefore, \(h\) satisfies \eqref{H-pb}.

We next derive the asymptotic behavior of \(h\) and \(h'\) as \(r\to\pi^-\). For positive functions \(F\) and \(G\), we shall write
\[
F(r)\asymp G(r)
\qquad \text{as } r\to\pi^-,
\]
if there exist constants \(c,C>0\) such that
\[
cG(r)\le F(r)\le CG(r)
\]
for all \(r\) sufficiently close to \(\pi\).

From the Frobenius expansion of \(h_\pi\), we have
\[
h_\pi(r)\asymp(\pi-r)^{\mu_+}
\qquad \text{as } r\to\pi^-.
\]
Therefore,
\[
\frac{1}{(\sin r)^{N-1}h_\pi(r)^2}
\asymp
(\pi-r)^{1-N-2\mu_+}
\qquad \text{as } r\to\pi^-.
\]
Since
\[
2-N-2\mu_+<0,
\]
we obtain
\[
\int_\delta^r
\frac{ds}{(\sin s)^{N-1}h_\pi(s)^2}
\asymp
(\pi-r)^{2-N-2\mu_+}
\qquad \text{as } r\to\pi^-.
\]
Hence, by \eqref{h-def},
\[
h(r)
\asymp
(\pi-r)^{2-N-\mu_+}
\qquad \text{as } r\to\pi^-.
\]
Since
\[
\mu_++\mu_-=2-N,
\]
we conclude that
\[
h(r)\asymp(\pi-r)^{\mu_-}
\qquad \text{as } r\to\pi^-.
\]

We now estimate \(h'\). Differentiating \eqref{h-def}, we obtain
\[
h'(r)
=
h_\pi'(r)
\int_\delta^r
\frac{ds}{(\sin s)^{N-1}h_\pi(s)^2}
+
\frac{1}{(\sin r)^{N-1}h_\pi(r)}.
\]
Moreover, from the Frobenius expansion,
\[
|h_\pi'(r)|\asymp(\pi-r)^{\mu_+-1}
\qquad \text{as } r\to\pi^-.
\]
Using the previous estimates, we obtain
\[
|h'(r)|
\le
C(\pi-r)^{1-N-\mu_+}
\qquad \text{as } r\to\pi^-,
\]
for some constant \(C>0\). Since
\[
1-N-\mu_+=\mu_- -1,
\]
it follows that
\[
|h'(r)|
\le
C(\pi-r)^{\mu_- -1}
\qquad \text{as } r\to\pi^-.
\]

\noindent\textbf{Case 2.} \(\lambda=\lambda^*\). In this case, the indicial equation has the double root
\[
\mu_+=\mu_-=-\frac{N-2}{2}.
\]
By the Frobenius theorem, there exists a solution \(h_\pi\) of the radial equation, defined near \(r=\pi\), such that
\[
h_\pi(r)
=
(\pi-r)^{-\frac{N-2}{2}}
\sum_{k=0}^{\infty}A_k(\pi-r)^k,
\qquad A_0>0.
\]
Hence, \(h_\pi>0\) near \(r=\pi\). We define \(a_{\lambda^*}\) and then \(h\) by \eqref{h-def}, exactly as in Case~1. The same positivity argument gives
\[
h(r)>0
\qquad \text{for } \delta<r<\pi,
\]
provided \(a_{\lambda^*}<\delta<\pi\). Moreover, the reduction-of-order formula gives
\[
h(r)\asymp
(\pi-r)^{-\frac{N-2}{2}}
\ln \frac{1}{\pi-r}
\qquad \text{as } r\to\pi^-,
\]
and
\[
|h'(r)|
\le
C(\pi-r)^{-\frac{N}{2}}
\ln\frac{1}{\pi-r}
\qquad \text{as } r\to\pi^-.
\]

The preceding analysis provides the radial Hardy barrier needed in the sequel.

\begin{Lem}\label{L3.1}
There exists \(a_\lambda\in[0,\pi)\) such that, for every
\(\delta\in(a_\lambda,\pi)\), there exists a radial function \(H=H_\delta\), \(H(x)=h(r)\),
satisfying \eqref{H-OK}. Moreover, as \(r\to\pi^-\), the following estimates hold:
\begin{itemize}
\item[\rm(i)] If \(0<\lambda<\lambda^*\), then
\[
h(r)\asymp(\pi-r)^{\mu_-},
\qquad
|h'(r)|\le C(\pi-r)^{\mu_- -1}.
\]

\item[\rm(ii)] If \(\lambda=\lambda^*\), then
\[
h(r)\asymp
(\pi-r)^{-\frac{N-2}{2}}
\ln\frac{1}{\pi-r},
\qquad
|h'(r)|
\le
C(\pi-r)^{-\frac{N}{2}}
\ln\frac{1}{\pi-r}.
\]
\end{itemize}
Here, \(C>0\) is independent of \(r\).
\end{Lem}

\subsection{Construction of suitable test functions}

We now construct a family of admissible test functions adapted to the Hardy barrier obtained in Lemma~\ref{L3.1}.

Throughout the rest of the paper, the notation $\tau\gg1$ means that $\tau>0$ is chosen sufficiently large.

Let \(\upzeta\in C_c^\infty(0,\infty)\) be such that
\[
\upzeta\ge 0,\qquad \upzeta\not\equiv0,
\qquad
\supp\upzeta\subset(0,1).
\]
For \(T,m\gg 1\), define
\[
\upzeta_T(t)=\upzeta^m\left(\frac{t}{T}\right),
\qquad t>0.
\]

We next consider a cutoff function \(\psi\in C^\infty([0,\infty))\) satisfying
\[
0\le \psi\le 1,\qquad
\psi=0 \ \text{on }[0,1/2],
\qquad
\psi=1 \ \text{on }[1,\infty).
\]
For \(\delta>a_\lambda\), let \(H(x)=h(r)\) be the radial Hardy barrier given by Lemma~\ref{L3.1}. For \(R,m\gg1\), define
\[
\psi_R(x)=H(x)\psi^m\bigl(R(\pi-r)\bigr),
\qquad \delta\le r=d(o,x)<\pi.
\]

For the critical case, we use a logarithmic cutoff near the singular point \(r=\pi\).
Let \(\eta\in C^\infty(\mathbb{R})\) be such that
\[
0\le \eta\le 1,\qquad
\eta=0 \ \text{on }(-\infty,0],
\qquad
\eta=1 \ \text{on }[1,\infty).
\]
For \(R,m\gg1\), define
\[
\eta_R(x)
=
H(x)\eta^m\left(
\frac{\ln\bigl(R(\pi-r)\bigr)}{\ln\sqrt R}
\right),
\qquad
\delta\le r=d(o,x)<\pi.
\]

Combining the temporal cutoff with the spatial cutoffs, we set, for \(T,R,m\gg1\),
\begin{equation}\label{test1}
\xi(t,x)
=
\upzeta_T(t)\psi_R(x),
\qquad (t,x)\in Q.
\end{equation}
In the critical case, we shall instead use
\begin{equation}\label{test2}
\xi_{\mathrm{crit}}(t,x)
=
\upzeta_T(t)\eta_R(x),
\qquad (t,x)\in Q.
\end{equation}

\begin{Lem}\label{lem:test-admissible}
For \(T,R,m\gg1\), the function \(\xi\) defined by \eqref{test1} belongs to \(\mathcal{X}\).
\end{Lem}

\begin{proof}
Since \(\supp\upzeta_T\subset(0,T)\) and \(\psi_R\) vanishes near \(r=\pi\), the function \(\xi\) has compact support in \(Q\). Moreover, the smoothness of \(\upzeta\), \(\psi\), and \(H\) gives
\[
\xi\in C_c^{1,2}(Q).
\]

By Lemma~\ref{L3.1}, we have
\[
H>0 \text{ in }\Omega_\delta,\qquad
H=0 \text{ on } \Gamma_\delta,\qquad
\partial_\nu H<0 \text{ on } \Gamma_\delta.
\]
In particular, \(H\ge0\) on \(\overline{\Omega_\delta}^{\,\pi}\). Since
\[
\upzeta_T\ge0
\qquad\text{and}\qquad
\psi^m\bigl(R(\pi-r)\bigr)\ge0,
\]
we obtain
\[
\xi\ge0.
\]
Furthermore, since \(H=0\) on \(\Gamma_\delta\), we have
\[
\xi=0
\qquad \text{on } \Sigma.
\]

Finally, for \(R\gg1\), we have \(R(\pi-\delta)\ge1\), and hence
\[
\psi^m(R(\pi-\delta))=1.
\]
Thus, on \(\Sigma\),
\[
\partial_\nu \xi
=
\upzeta_T(t)\,\partial_\nu H
\le0.
\]
Therefore, \(\xi\in\mathcal{X}\).
\end{proof}

\begin{Lem}\label{lem:test-critical-admissible}
For \(T,R,m\gg1\), the function \(\xi_{\mathrm{crit}}\) defined by \eqref{test2} belongs to \(\mathcal{X}\).
\end{Lem}

\begin{proof}
The proof is analogous to that of Lemma~\ref{lem:test-admissible}. Indeed, \(\supp\upzeta_T\subset(0,T)\), and \(\eta_R\) vanishes near \(r=\pi\). Hence, \(\xi_{\mathrm{crit}}\) has compact support in \(Q\). Moreover, \(\xi_{\mathrm{crit}}\ge0\), and since \(H=0\) on \(\Gamma_\delta\), we have
\[
\xi_{\mathrm{crit}}=0
\qquad \text{on } \Sigma.
\]
Finally, for \(R\gg1\),
\[
\frac{\ln\bigl(R(\pi-\delta)\bigr)}{\ln\sqrt R}\ge1,
\]
and hence
\[
\eta^m\left(
\frac{\ln\bigl(R(\pi-\delta)\bigr)}{\ln\sqrt R}
\right)=1.
\]
Therefore, on \(\Sigma\),
\[
\partial_\nu \xi_{\mathrm{crit}}
=
\upzeta_T(t)\,\partial_\nu H
\le0.
\]
Thus, \(\xi_{\mathrm{crit}}\in\mathcal{X}\).
\end{proof}

\subsection{Auxiliary estimates}

In this subsection, we establish the auxiliary estimates needed in the proof of the main result.

Throughout the rest of the paper, \(C\) denotes a positive constant independent of the scaling parameters \(T\) and \(R\). Its value may change from line to line.

In this subsection, let \(\delta\in(a_\lambda,\pi)\), and let \(H(x)=h(r)\) be the radial Hardy barrier given by Lemma~\ref{L3.1}.

\begin{Lem}\label{L3.4}
For \(R,m\gg1\), the following estimates hold:
\[
\int_{\Omega_\delta}
(\sin r)^{-\frac{\alpha}{p-1}}\psi_R(x)\,d\mu
\le
C\ln R \left(
\ln R+
R^{\frac{\alpha}{p-1}-N-\mu_-}
\right),
\]
and, for the logarithmic cutoff,
\[
\int_{\Omega_\delta}
(\sin r)^{-\frac{\alpha}{p-1}}\eta_R(x)\,d\mu
\le
C\ln R \left(
\ln R+
R^{\frac{\alpha}{p-1}-N-\mu_-}
\right).
\]
\end{Lem}

\begin{proof}
Let
\[
I_R=
\int_{\Omega_\delta}
(\sin r)^{-\frac{\alpha}{p-1}}\psi_R(x)\,d\mu.
\]
Since
\[
\supp \psi_R
\subset
\left\{
x\in \mathbb S^N:
\delta\le r\le \pi-\frac{1}{2R}
\right\},
\]
using geodesic polar coordinates on \(\mathbb S^N\) and the definition of \(\psi_R\), we obtain
\[
\begin{aligned}
I_R
&\le
C\int_\delta^{\pi-\frac{1}{2R}}
(\sin r)^{N-1-\frac{\alpha}{p-1}}
h(r)\psi^m(R(\pi-r))\,dr\\
&\le C\int_\delta^{\pi-\frac{1}{2R}}
(\sin r)^{N-1-\frac{\alpha}{p-1}}
h(r)\,dr.
\end{aligned}
\]
Choose \(r_0\in(\delta,\pi)\) sufficiently close to \(\pi\). Then
\begin{equation}\label{s1-est}
I_R \le C\left(
\int_\delta^{r_0}
(\sin r)^{N-1-\frac{\alpha}{p-1}}h(r)\,dr
+
\int_{r_0}^{\pi-\frac{1}{2R}}
(\sin r)^{N-1-\frac{\alpha}{p-1}}h(r)\,dr
\right).
\end{equation}
On \([\delta,r_0]\), the integrand is bounded. Hence,
\begin{equation}\label{s2-est}
\int_\delta^{r_0}
(\sin r)^{N-1-\frac{\alpha}{p-1}}h(r)\,dr
\le C.
\end{equation}
We now estimate the integral over \((r_0,\pi-1/(2R))\). By Lemma~\ref{L3.1}, for \(0<\lambda\le \lambda^*\), we have
\[
h(r)\le C(\pi-r)^{\mu_-}\ln\frac{1}{\pi-r}
\qquad \text{as } r\to\pi^-.
\]
Moreover, \(\sin r\asymp \pi-r\) as \(r\to\pi^-\). Therefore,
\begin{equation}\label{s3-est}
\begin{aligned}
\int_{r_0}^{\pi-\frac{1}{2R}}
(\sin r)^{N-1-\frac{\alpha}{p-1}}h(r)\,dr
&\le C\int_{r_0}^{\pi-\frac{1}{2R}}
(\pi-r)^{N-1-\frac{\alpha}{p-1}+\mu_-}
\ln\frac{1}{\pi-r}\,dr\\
&\le C\ln R \left(
\ln R+
R^{\frac{\alpha}{p-1}-N-\mu_-}
\right).
\end{aligned}
\end{equation}
Combining \eqref{s1-est}, \eqref{s2-est}, and \eqref{s3-est}, we obtain the first claimed estimate.

The same argument applies to \(\eta_R\). Indeed, since \(\eta_R=0\) whenever
\[
R(\pi-r)\le 1,
\]
or equivalently whenever
\[
\pi-\frac1R\le r<\pi,
\]
the preceding proof can be repeated with \(\pi-\frac1R\) in place of \(\pi-\frac1{2R}\). The resulting estimate is unchanged.
\end{proof}

\begin{Lem}\label{L3.5}
For \(R,m\gg1\), the following estimate holds:
\[
\int_{\Omega_\delta}
(\sin r)^{-\frac{\alpha}{p-1}}
\psi_R(x)^{-\frac1{p-1}}
\left|L_\lambda\psi_R(x)\right|^{\frac{p}{p-1}}
\,d\mu
\le
C(\ln R)^{\frac{p}{p-1}}\,
R^{\frac{2p}{p-1}-N+\frac{\alpha}{p-1}-\mu_-}.
\]
\end{Lem}

\begin{proof}
Set
\[
J_R=\int_{\Omega_\delta}
(\sin r)^{-\frac{\alpha}{p-1}}
\psi_R(x)^{-\frac1{p-1}}
\left|L_\lambda\psi_R(x)\right|^{\frac{p}{p-1}}
\,d\mu.
\]
We write
\[
\psi_R(x)=H(x)\chi_R(r),
\qquad
\chi_R(r)=\psi^m(R(\pi-r)).
\]
Since \(L_\lambda H=0\), by Lemma~\ref{L3.1}, we have
\[
L_\lambda\psi_R
=
H\Delta_{\mathbb S^N}\chi_R
+
2\left\langle \nabla_{\mathbb S^N}H,\nabla_{\mathbb S^N}\chi_R\right\rangle.
\]
Since all functions are radial, this identity becomes
\begin{equation}\label{Llambda}
L_\lambda\psi_R
=
h(r)\left(\chi_R''(r)+(N-1)\cot r\,\chi_R'(r)\right)
+
2h'(r)\chi_R'(r).
\end{equation}
Moreover,
\[
\supp \chi_R'\cup\supp \chi_R''
\subset
\left\{x\in \mathbb S^N:
\pi-\frac1R \le r\le \pi-\frac1{2R}
\right\}.
\]
Therefore, using geodesic polar coordinates on \(\mathbb S^N\), we obtain
\begin{equation}\label{hot}
\begin{aligned}
J_R
&\le
C \int_{\pi-\frac1R}^{\pi-\frac1{2R}}
(\sin r)^{N-1-\frac{\alpha}{p-1}}
h(r)^{-\frac1{p-1}}\chi_R(r)^{-\frac1{p-1}}
\left|L_\lambda\psi_R(x)\right|^{\frac{p}{p-1}}
\,dr .
\end{aligned}
\end{equation}

For \(R\gg1\) and \(\pi-\frac1R<r<\pi-\frac1{2R}\), we have
\[
\sin r\asymp \pi-r\asymp \frac1R,
\qquad
|\cot r|\le CR.
\]
Moreover,
\[
|\chi_R'(r)|\le CR\psi^{m-1}(R(\pi-r)),
\qquad
|\chi_R''(r)|\le CR^2\psi^{m-2}(R(\pi-r)).
\]
By Lemma~\ref{L3.1}, in the same range of \(r\),
\[
h(r)\le C R^{-\mu_-}\ln R,
\qquad
|h'(r)|\le C R^{1-\mu_-}\ln R.
\]
Hence, from \eqref{Llambda},
\[
|L_\lambda\psi_R(x)|
\le
C R^{2-\mu_-}\ln R\,\psi^{m-2}(R(\pi-r)).
\]
Also, Lemma~\ref{L3.1} gives
\[
h(r)^{-\frac1{p-1}}
\le
C R^{\frac{\mu_-}{p-1}}.
\]
Consequently,
\[
h(r)^{-\frac1{p-1}}
\left|L_\lambda\psi_R(x)\right|^{\frac{p}{p-1}}
\le
C
R^{\frac{\mu_-}{p-1}+\frac{p}{p-1}(2-\mu_-)}
(\ln R)^{\frac{p}{p-1}}
\psi^{(m-2)\frac{p}{p-1}}(R(\pi-r)).
\]
Since
\[
\chi_R(r)=\psi^m(R(\pi-r)),
\]
we have
\[
\chi_R(r)^{-\frac1{p-1}}
\psi^{(m-2)\frac{p}{p-1}}(R(\pi-r))
=
\psi^{\frac{m(p-1)-2p}{p-1}}(R(\pi-r)).
\]
Thus, if
\[
m\ge \frac{2p}{p-1},
\]
then
\[
\chi_R(r)^{-\frac1{p-1}}
\psi^{(m-2)\frac{p}{p-1}}(R(\pi-r))
\le 1.
\]
Therefore, using again \(\sin r\asymp R^{-1}\), we obtain from \eqref{hot}
\[
J_R
\le
C(\ln R)^{\frac{p}{p-1}}
R^{1+\frac{2p}{p-1}-N+\frac{\alpha}{p-1}-\mu_-}
\int_{\pi-\frac1R}^{\pi-\frac1{2R}}dr.
\]
Since the length of the interval is of order \(R^{-1}\), we get
\[
J_R
\le
C(\ln R)^{\frac{p}{p-1}}
R^{\frac{2p}{p-1}-N+\frac{\alpha}{p-1}-\mu_-}.
\]
This completes the proof.
\end{proof}

\begin{Lem}\label{L3.6}
Assume that
\begin{equation}\label{cd-pcrit}
\frac{2p}{p-1}-N+\frac{\alpha}{p-1}-\mu_-=0.
\end{equation}
For \(R,m\gg1\), the following estimates hold:
\begin{itemize}
\item[\rm(i)] If \(0<\lambda<\lambda^*\), then
\[
\int_{\Omega_\delta}
(\sin r)^{-\frac{\alpha}{p-1}}
\eta_R(x)^{-\frac1{p-1}}
\left|L_\lambda\eta_R(x)\right|^{\frac{p}{p-1}}
\,d\mu
\le
C(\ln R)^{1-\frac{p}{p-1}}.
\]

\item[\rm(ii)] If \(\lambda=\lambda^*\), then
\[
\int_{\Omega_\delta}
(\sin r)^{-\frac{\alpha}{p-1}}
\eta_R(x)^{-\frac1{p-1}}
\left|L_\lambda\eta_R(x)\right|^{\frac{p}{p-1}}
\,d\mu
\le
C(\ln R)^{2-\frac{p}{p-1}}.
\]
\end{itemize}
\end{Lem}

\begin{proof}
We only indicate the changes with respect to the proof of Lemma~\ref{L3.5}. Write
\[
\eta_R(x)=H(x)\theta_R(r),
\qquad
\theta_R(r)=
\eta^m\left(
\frac{\ln(R(\pi-r))}{\ln\sqrt R}
\right).
\]
Then
\[
\supp \theta_R'\cup\supp \theta_R''
\subset
\left\{
x\in\mathbb S^N:
\pi-\frac1{\sqrt R}<r<\pi-\frac1R
\right\}.
\]
Moreover, on this set,
\[
|\theta_R'(r)|
\le
\frac{C}{(\pi-r)\ln R}
\eta^{m-1}\left(
\frac{\ln(R(\pi-r))}{\ln\sqrt R}
\right),
\]
and
\[
|\theta_R''(r)|
\le
\frac{C}{(\pi-r)^2\ln R}
\eta^{m-2}\left(
\frac{\ln(R(\pi-r))}{\ln\sqrt R}
\right).
\]
Repeating the argument of Lemma~\ref{L3.5}, and using \(m\gg1\), we obtain
\[
\begin{aligned}
&\int_{\Omega_\delta}
(\sin r)^{-\frac{\alpha}{p-1}}
\eta_R^{-\frac1{p-1}}
|L_\lambda\eta_R|^{\frac{p}{p-1}}
\,d\mu \\
&\qquad\le
\frac{C}{(\ln R)^{\frac{p}{p-1}}}
\int_{\pi-\frac1{\sqrt R}}^{\pi-\frac1R}
(\pi-r)^{N-1-\frac{\alpha}{p-1}+\mu_- -\frac{2p}{p-1}}
\,\Lambda_\lambda(r)\,dr,
\end{aligned}
\]
where
\[
\Lambda_\lambda(r)=
\begin{cases}
1, & 0<\lambda<\lambda^*,\\[4pt]
\displaystyle \ln\frac1{\pi-r}, & \lambda=\lambda^*.
\end{cases}
\]
By \eqref{cd-pcrit},
\[
N-1-\frac{\alpha}{p-1}+\mu_- -\frac{2p}{p-1}=-1.
\]
Therefore, if \(0<\lambda<\lambda^*\), then
\[
\int_{\pi-\frac1{\sqrt R}}^{\pi-\frac1R}
\frac{dr}{\pi-r}
\le C\ln R,
\]
and hence
\[
\int_{\Omega_\delta}
(\sin r)^{-\frac{\alpha}{p-1}}
\eta_R^{-\frac1{p-1}}
|L_\lambda\eta_R|^{\frac{p}{p-1}}
\,d\mu
\le C(\ln R)^{1-\frac{p}{p-1}}.
\]
If \(\lambda=\lambda^*\), then
\[
\int_{\pi-\frac1{\sqrt R}}^{\pi-\frac1R}
\frac1{\pi-r}\ln\frac1{\pi-r}\,dr
\le C(\ln R)^2,
\]
and therefore
\[
\int_{\Omega_\delta}
(\sin r)^{-\frac{\alpha}{p-1}}
\eta_R^{-\frac1{p-1}}
|L_\lambda\eta_R|^{\frac{p}{p-1}}
\,d\mu
\le C(\ln R)^{2-\frac{p}{p-1}}.
\]
\end{proof}

The following estimates follow directly from the definition of the cutoff function \(\upzeta_T\).

\begin{Lem}\label{L3.7}
For \(T,m\gg1\), we have
\[
\int_0^\infty \upzeta_T(t)\,dt = C_{\upzeta,m} T,
\qquad
C_{\upzeta,m}=\int_0^\infty \upzeta^m(s)\,ds>0,
\]
and
\[
\int_0^\infty
\upzeta_T(t)^{-\frac1{p-1}}
|\upzeta_T'(t)|^{\frac{p}{p-1}}\,dt
\le
CT^{-\frac1{p-1}}.
\]
\end{Lem}

\section{Proof of the main result}\label{sec4}

This section is devoted to the proof of Theorem~\ref{T1.2}. We first prove the nonexistence assertions and then establish the existence result.

\begin{proof}[Proof of Theorem~\ref{T1.2}]
\noindent\textbf{Part (i).}
We argue by contradiction and suppose that \eqref{P} admits a weak solution
\(u \in L^p_{\mathrm{loc}}(Q)\) for some function
\(f\in L^1_{\mathrm{loc}}(\overline{\Omega_\delta}^{\,\pi})\) satisfying
\(f\ge0\) and \(f\not\equiv0\).
Let \(\xi \in \mathcal{X}\). By Definition~\ref{def-ws}, we have
\[
\int_{Q} (\sin r)^{\alpha} |u|^p \, \xi \, d\mu \, dt
+ \int_{Q} f(x) \, \xi \, d\mu \, dt
\le
\int_{Q} |u| \, |\partial_t \xi| \, d\mu \, dt
+ \int_{Q} |u|\, |L_\lambda \xi|\,d\mu \, dt.
\]
By Young's inequality,
\[
\int_{Q} |u| \, |\partial_t \xi| \, d\mu \, dt
\le
\frac{1}{2}\int_{Q} (\sin r)^{\alpha} |u|^p \, \xi \, d\mu \, dt
+
C\int_{Q} (\sin r)^{-\frac{\alpha}{p-1}}
|\partial_t \xi|^{\frac{p}{p-1}}
\xi^{-\frac{1}{p-1}}\,d\mu \, dt,
\]
and
\[
\int_{Q} |u|\, |L_\lambda \xi|\,d\mu \, dt
\le
\frac{1}{2}\int_{Q} (\sin r)^{\alpha} |u|^p \, \xi \, d\mu \, dt
+
C\int_{Q} (\sin r)^{-\frac{\alpha}{p-1}}
|L_\lambda \xi|^{\frac{p}{p-1}}
\xi^{-\frac{1}{p-1}}\,d\mu \, dt.
\]
Combining the above inequalities, we obtain
\begin{equation}\label{ET1}
\int_{Q} f(x) \, \xi \, d\mu \, dt
\le
C \left(K_1(\xi)+K_2(\xi)\right),
\end{equation}
where
\[
K_1(\xi)=
\int_{Q} (\sin r)^{-\frac{\alpha}{p-1}}
|\partial_t \xi|^{\frac{p}{p-1}}
\xi^{-\frac{1}{p-1}}\,d\mu \, dt,
\]
and
\[
K_2(\xi)=
\int_{Q} (\sin r)^{-\frac{\alpha}{p-1}}
|L_\lambda \xi|^{\frac{p}{p-1}}
\xi^{-\frac{1}{p-1}}\,d\mu \, dt,
\]
provided that \(K_i(\xi)<\infty\) for \(i=1,2\).

Let \(a_\lambda\in [0,\pi)\) be given by \eqref{alambda-def}, with \(\delta\in(a_\lambda,\pi)\).

\noindent\textbf{Case (a).} Assume that \(p\) satisfies \eqref{cd-nnex1}. For \(T,R,m\gg1\), let \(\xi\) be the test function given by \eqref{test1}. By Lemma~\ref{lem:test-admissible}, we have \(\xi\in\mathcal{X}\). Moreover,
\[
K_1(\xi)=
\left(\int_0^\infty
\upzeta_T(t)^{-\frac1{p-1}}
|\upzeta_T'(t)|^{\frac{p}{p-1}}\,dt\right)
\left(\int_{\Omega_\delta}
(\sin r)^{-\frac{\alpha}{p-1}}\psi_R(x)\,d\mu\right),
\]
and
\[
K_2(\xi)=
\left(\int_0^\infty \upzeta_T(t)\,dt\right)
\left(\int_{\Omega_\delta}
(\sin r)^{-\frac{\alpha}{p-1}}
\psi_R(x)^{-\frac1{p-1}}
\left|L_\lambda\psi_R(x)\right|^{\frac{p}{p-1}}
\,d\mu\right).
\]

Using the first estimate in Lemma~\ref{L3.4} and the second estimate in Lemma~\ref{L3.7}, we obtain
\begin{equation}\label{K1est-a}
K_1(\xi)
\le
CT^{-\frac1{p-1}}\ln R \left(
\ln R+
R^{\frac{\alpha}{p-1}-N-\mu_-}
\right).
\end{equation}
Moreover, by the first identity in Lemma~\ref{L3.7} and Lemma~\ref{L3.5},
\begin{equation}\label{K2est-a}
K_2(\xi)
\le
CT(\ln R)^{\frac{p}{p-1}}\,
R^{\frac{2p}{p-1}-N+\frac{\alpha}{p-1}-\mu_-}.
\end{equation}

On the other hand, since
\[
f\ge0,\qquad f\not\equiv0,
\]
and since, by Lemma~\ref{L3.1}, \(H>0\) in \(\Omega_\delta\), there exist
\[
\delta<r_1<r_2<\pi
\]
such that
\[
\int_{\{r_1<r<r_2\}} f(x)H(x)\,d\mu>0.
\]
Moreover, for \(R\gg1\), we have
\[
R(\pi-r)\ge1
\qquad \text{for } r_1\le r\le r_2.
\]
Hence,
\[
\psi^m(R(\pi-r))=1
\qquad \text{on } \{r_1\le r\le r_2\}.
\]
Therefore,
\[
\begin{aligned}
\int_Q f(x)\xi(t,x)\,d\mu\,dt
&=
\left(\int_0^\infty \upzeta_T(t)\,dt\right)
\left(\int_{\Omega_\delta} f(x)H(x)\psi^m(R(\pi-r))\,d\mu\right)\\
&\ge
\left(\int_0^\infty \upzeta_T(t)\,dt\right)
\left(\int_{\{r_1<r<r_2\}} f(x)H(x)\,d\mu\right).
\end{aligned}
\]
Using the first identity in Lemma~\ref{L3.7}, we obtain
\begin{equation}\label{intf-a}
\int_Q f(x)\xi(t,x)\,d\mu\,dt \ge CT.
\end{equation}

Combining \eqref{ET1}, \eqref{K1est-a}, \eqref{K2est-a}, and \eqref{intf-a}, we obtain
\[
T\le C\left[
T^{-\frac1{p-1}}\ln R \left(
\ln R+
R^{\frac{\alpha}{p-1}-N-\mu_-}
\right)
+
T(\ln R)^{\frac{p}{p-1}}\,
R^{\frac{2p}{p-1}-N+\frac{\alpha}{p-1}-\mu_-}
\right].
\]
Dividing by \(T\), we get
\[
1\le C\left[
T^{-\frac p{p-1}}\ln R \left(
\ln R+
R^{\frac{\alpha}{p-1}-N-\mu_-}
\right)
+
(\ln R)^{\frac{p}{p-1}}\,
R^{\frac{2p}{p-1}-N+\frac{\alpha}{p-1}-\mu_-}
\right].
\]
Choose \(T=R^\theta\), where
\[
\theta>\max\left\{0,\frac{\alpha-(N+\mu_-)(p-1)}{p}\right\}.
\]
Since
\[
N-2+\mu_-=\lambda_N,
\]
condition \eqref{cd-nnex1} is equivalent to
\[
p>1,
\qquad
\frac{2p}{p-1}-N+\frac{\alpha}{p-1}-\mu_-<0.
\]
Therefore, letting \(R\to\infty\), the right-hand side tends to \(0\), which is impossible. This contradiction proves the nonexistence assertion in Case~(a).

\noindent\textbf{Case (b).}
Let \(\alpha>-2\), and let \(p\) satisfy \eqref{cd-nnex1-crit}. We note that \eqref{cd-nnex1-crit} is equivalent to \eqref{cd-pcrit}. For \(T,R,m\gg1\), let \(\xi_{\mathrm{crit}}\) be the test function given by \eqref{test2}. By Lemma~\ref{lem:test-critical-admissible}, we have \(\xi_{\mathrm{crit}}\in\mathcal{X}\). Moreover, we have 
\[
K_1(\xi_{\mathrm{crit}})=
\left(\int_0^\infty
\upzeta_T(t)^{-\frac1{p-1}}
|\upzeta_T'(t)|^{\frac{p}{p-1}}\,dt\right)
\left(\int_{\Omega_\delta}
(\sin r)^{-\frac{\alpha}{p-1}}\eta_R(x)\,d\mu\right),
\]
and
\[
K_2(\xi_{\mathrm{crit}})=
\left(\int_0^\infty \upzeta_T(t)\,dt\right)
\left(\int_{\Omega_\delta}
(\sin r)^{-\frac{\alpha}{p-1}}
\eta_R(x)^{-\frac1{p-1}}
\left|L_\lambda\eta_R(x)\right|^{\frac{p}{p-1}}
\,d\mu\right).
\]

Using the second estimate in Lemma~\ref{L3.4} and the second estimate in Lemma~\ref{L3.7}, we obtain
\begin{equation}\label{est-K1-crit-OK}
K_1(\xi_{\mathrm{crit}})
\le
CT^{-\frac1{p-1}}\ln R \left(
\ln R+
R^{\frac{\alpha}{p-1}-N-\mu_-}
\right).
\end{equation}
Moreover, by Lemma~\ref{L3.6} and the first identity in Lemma~\ref{L3.7}, we obtain
\begin{equation}\label{est-K2-crit-OK}
K_2(\xi_{\mathrm{crit}})
\le
CT(\ln R)^{k_\lambda-\frac{p}{p-1}},
\end{equation}
where
\[
k_\lambda=
\begin{cases}
1, & \text{if } 0<\lambda<\lambda^*,\\
2, & \text{if } \lambda=\lambda^*.
\end{cases}
\]

On the other hand, as in Case~(a), using the fact that the logarithmic cutoff is equal to \(1\) on every fixed annulus \(r_1\le r\le r_2<\pi\) for \(R\gg1\), we obtain
\begin{equation}\label{intf-b-crit}
\int_Q f(x)\xi_{\mathrm{crit}}(t,x)\,d\mu\,dt \ge CT.
\end{equation}

Combining \eqref{ET1} with \(\xi=\xi_{\mathrm{crit}}\), together with
\eqref{est-K1-crit-OK}, \eqref{est-K2-crit-OK}, and \eqref{intf-b-crit}, we obtain
\[
1 \le C\left[
T^{-\frac p{p-1}}\ln R \left(
\ln R+
R^{\frac{\alpha}{p-1}-N-\mu_-}
\right)
+
(\ln R)^{k_\lambda-\frac{p}{p-1}}
\right].
\]
Under assumption \eqref{cd-lambda-alpha}, we have
\[
k_\lambda-\frac{p}{p-1}<0.
\]
We now take \(T=R^\theta\), with \(\theta>0\) sufficiently large, and let \(R\to\infty\). The right-hand side tends to \(0\), which is impossible. This contradiction proves Case~(b).

This completes the proof of part~(i).

\noindent\textbf{Part (ii).}
Assume now that \eqref{cd-ex1} holds. We consider separately the cases
\(0<\lambda<\lambda^*\) and \(\lambda=\lambda^*\).

\noindent\textbf{Case 1.} \(0<\lambda<\lambda^*\). Since \eqref{cd-ex1} holds and
\[
\lambda_N=-\mu_+,
\]
we have
\[
1<p<1+\frac{\alpha+2}{-\mu_+}.
\]
Equivalently,
\[
-\frac{\alpha+2}{p-1}<\mu_+.
\]
Since \(\mu_-<\mu_+\), we may choose \(q\) such that
\begin{equation}\label{choice-q-ex}
\max\left\{\mu_-,-\frac{\alpha+2}{p-1}\right\}<q<\mu_+.
\end{equation}

Let
\[
G_q(x)=(\sin r)^q,
\qquad 0<r=d(o,x)<\pi.
\]
Since \(r=d(o,x)\) is smooth away from \(o\) and its antipodal point, we have
\(G_q\in C^\infty(\Omega_\delta)\). A direct computation gives, for every
\(x\in\Omega_\delta\),
\[
\Delta_{\mathbb S^N}G_q(x)
=
\left[
\frac{q(q+N-2)}{\sin^2 r}
-
q(q+N-1)
\right]G_q(x).
\]
Therefore,
\[
-\Delta_{\mathbb S^N}G_q(x)
-\frac{\lambda}{\sin^2 r}G_q(x)
=
\left[
\frac{A_q}{\sin^2 r}
+
q(q+N-1)
\right]G_q(x),
\]
where
\[
A_q=-(q(q+N-2)+\lambda).
\]
Since \(q\in(\mu_-,\mu_+)\), we have \(A_q>0\). Consequently, there exists
\(b=b_q\in[0,\pi)\) such that
\[
\frac{A_q}{\sin^2 r}
+
q(q+N-1)>0
\qquad \text{for } b<r<\pi.
\]
We now take
\[
\delta\in(b,\pi).
\]
Then, for every \(x\in\Omega_\delta\),
\[
-\Delta_{\mathbb S^N}G_q(x)
-\frac{\lambda}{\sin^2 r}G_q(x)>0.
\]
Moreover, as \(r\to\pi^-\),
\[
-\Delta_{\mathbb S^N}G_q(x)
-\frac{\lambda}{\sin^2 r}G_q(x)
\asymp
(\sin r)^{q-2}.
\]
On the other hand,
\[
(\sin r)^\alpha G_q(x)^p
=
(\sin r)^{\alpha+pq}.
\]
By \eqref{choice-q-ex},
\[
q>-\frac{\alpha+2}{p-1},
\]
which is equivalent to
\[
q-2<\alpha+pq.
\]
Thus,
\[
\frac{(\sin r)^\alpha G_q(x)^p}
{-\Delta_{\mathbb S^N}G_q(x)-\dfrac{\lambda}{\sin^2 r}G_q(x)}
\to0^+
\qquad \text{as } r\to\pi^-.
\]
Since the denominator is positive in \(\Omega_\delta\), it follows that
\[
M_q
=
\sup_{x\in\Omega_\delta}
\frac{(\sin r)^\alpha G_q(x)^p}
{-\Delta_{\mathbb S^N}G_q(x)-\dfrac{\lambda}{\sin^2 r}G_q(x)}
<\infty.
\]

Choose \(\varepsilon>0\) so small that
\[
\varepsilon^{p-1}M_q\le \frac12.
\]
For \(x\in\overline{\Omega_\delta}^{\,\pi}\), define
\[
U_\varepsilon(x)=\varepsilon G_q(x)
=
\varepsilon(\sin r)^q.
\]
Then
\[
U_\varepsilon\in C^\infty(\Omega_\delta)\cap C(\overline{\Omega_\delta}^{\,\pi}),
\qquad
U_\varepsilon>0 \quad \text{in } \overline{\Omega_\delta}^{\,\pi}.
\]
Moreover, for every \(x\in\Omega_\delta\),
\[
\begin{aligned}
&-\Delta_{\mathbb S^N}U_\varepsilon(x)
-\frac{\lambda}{\sin^2 r}U_\varepsilon(x)
-(\sin r)^\alpha U_\varepsilon(x)^p \\
&=
\varepsilon\left(
-\Delta_{\mathbb S^N}G_q(x)
-\frac{\lambda}{\sin^2 r}G_q(x)
\right)
-\varepsilon^p(\sin r)^\alpha G_q(x)^p \\
&\ge
\frac{\varepsilon}{2}
\left(
-\Delta_{\mathbb S^N}G_q(x)
-\frac{\lambda}{\sin^2 r}G_q(x)
\right)>0.
\end{aligned}
\]
For \(x\in\Omega_\delta\), set
\[
F_\varepsilon(x)
=
-\Delta_{\mathbb S^N}U_\varepsilon(x)
-\frac{\lambda}{\sin^2 r}U_\varepsilon(x)
-(\sin r)^\alpha U_\varepsilon(x)^p.
\]
Then
\[
F_\varepsilon>0
\qquad \text{in } \Omega_\delta.
\]
We choose
\[
f_\varepsilon(x)=\frac12F_\varepsilon(x),
\qquad x\in\Omega_\delta.
\]
Thus, \(f_\varepsilon>0\) in \(\Omega_\delta\), and, for every \(x\in\Omega_\delta\),
\[
-\Delta_{\mathbb S^N}U_\varepsilon(x)
-\frac{\lambda}{\sin^2 r}U_\varepsilon(x)
\ge
(\sin r)^\alpha U_\varepsilon(x)^p+f_\varepsilon(x).
\]
Since \(r=\pi\) is not included in \(\overline{\Omega_\delta}^{\,\pi}\), and since \(F_\varepsilon\) is smooth for \(\delta<r<\pi\) and continuous up to \(r=\delta\), we have
\[
f_\varepsilon\in C(\overline{\Omega_\delta}^{\,\pi}).
\]
Moreover, by the choice of \(\delta>b\), \(f_\varepsilon>0\) on \(\overline{\Omega_\delta}^{\,\pi}\).

We now construct a time-dependent solution. Let
\[
a(t)=1-\frac12e^{-t},
\qquad t>0.
\]
Then
\begin{equation}\label{aff}
\frac12<a(t)<1,
\qquad
a'(t)>0,
\qquad t>0.
\end{equation}
Define
\[
u(t,x)=a(t)U_\varepsilon(x),
\qquad (t,x)\in Q.
\]
Since \(a\in C^\infty(0,\infty)\) and
\[
U_\varepsilon\in C^\infty(\Omega_\delta)\cap C(\overline{\Omega_\delta}^{\,\pi}),
\]
we have
\[
u\in C^\infty((0,\infty)\times\Omega_\delta)\cap C(Q).
\]
Furthermore,
\[
u\ge0
\qquad \text{on } (0,\infty)\times\Gamma_\delta.
\]

Finally, using, for every \(x\in\Omega_\delta\),
\[
-\Delta_{\mathbb S^N}U_\varepsilon(x)
-\frac{\lambda}{\sin^2 r}U_\varepsilon(x)
=
(\sin r)^\alpha U_\varepsilon(x)^p+F_\varepsilon(x),
\]
and \eqref{aff}, we obtain, for every \((t,x)\in(0,\infty)\times\Omega_\delta\),
\[
\begin{aligned}
&\partial_tu(t,x)-\Delta_{\mathbb S^N}u(t,x)
-\frac{\lambda}{\sin^2 r}u(t,x)
-(\sin r)^\alpha u(t,x)^p-f_\varepsilon(x) \\
&=
a'(t)U_\varepsilon(x)
+
\bigl(a(t)-a(t)^p\bigr)(\sin r)^\alpha U_\varepsilon(x)^p
+
\left(a(t)-\frac12\right)F_\varepsilon(x)\\
&\ge 0.
\end{aligned}
\]
Therefore,
\[
\partial_tu-\Delta_{\mathbb S^N}u-\frac{\lambda}{\sin^2 r}u
\ge
(\sin r)^\alpha u^p+f_\varepsilon
\qquad \text{in } (0,\infty)\times\Omega_\delta.
\]
Hence, \(u\) is a classical solution of \eqref{P} corresponding to the positive source \(f_\varepsilon\).

\noindent\textbf{Case 2.} \(\lambda=\lambda^*\). Choose
\[
0<\gamma<1.
\]
We define
\[
G_\gamma(x)=(\sin r)^{-\lambda_N}(-\ln(\sin r))^\gamma,
\qquad \frac{\pi}{2}<r=d(o,x)<\pi.
\]
Then \(G_\gamma>0\) and \(G_\gamma\in C^\infty\) on
\[
\left\{x\in\mathbb S^N: \frac{\pi}{2}<r=d(o,x)<\pi\right\}.
\]
A direct asymptotic computation gives, as \(r=d(o,x)\to\pi^-\),
\begin{equation}\label{Lg-asymp}
-\Delta_{\mathbb S^N}G_\gamma(x)
-\frac{\lambda^*}{\sin^2 r}G_\gamma(x)
\sim
\gamma(1-\gamma)(\sin r)^{-\lambda_N-2}
(-\ln(\sin r))^{\gamma-2}.
\end{equation}
Since \(0<\gamma<1\), the right-hand side in \eqref{Lg-asymp} is positive. Hence, there exists \(b=b_\gamma\in[\pi/2,\pi)\) such that
\[
-\Delta_{\mathbb S^N}G_\gamma(x)
-\frac{\lambda^*}{\sin^2 r}G_\gamma(x)>0
\]
for every \(x\in\mathbb S^N\) satisfying
\[
b<r=d(o,x)<\pi.
\]

We now take
\[
\delta\in(b,\pi).
\]
Then the above positivity holds for every \(x\in\Omega_\delta\). Moreover, for \(x\in\Omega_\delta\),
\[
(\sin r)^\alpha G_\gamma(x)^p
=
(\sin r)^{\alpha-\lambda_N p}
(-\ln(\sin r))^{\gamma p}.
\]
By \eqref{cd-ex1}, we have
\[
-\lambda_N-2<\alpha-\lambda_N p.
\]
Therefore, using \eqref{Lg-asymp},
\[
\frac{(\sin r)^\alpha G_\gamma(x)^p}
{-\Delta_{\mathbb S^N}G_\gamma(x)-\dfrac{\lambda^*}{\sin^2 r}G_\gamma(x)}
\to0^+
\qquad \text{as } r=d(o,x)\to\pi^-.
\]
Since the denominator is positive in \(\Omega_\delta\), it follows that
\[
M_\gamma
=
\sup_{x\in\Omega_\delta}
\frac{(\sin r)^\alpha G_\gamma(x)^p}
{-\Delta_{\mathbb S^N}G_\gamma(x)-\dfrac{\lambda^*}{\sin^2 r}G_\gamma(x)}
<\infty.
\]

Choose \(\varepsilon>0\) so small that
\[
\varepsilon^{p-1}M_\gamma\le \frac12.
\]
For \(x\in \overline{\Omega_\delta}^{\,\pi}\), define
\[
U_\varepsilon(x)=\varepsilon G_\gamma(x).
\]
Then
\[
U_\varepsilon\in C^\infty(\Omega_\delta)\cap C(\overline{\Omega_\delta}^{\,\pi}),
\qquad
U_\varepsilon>0
\quad \text{in } \overline{\Omega_\delta}^{\,\pi}.
\]
Moreover, for every \(x\in\Omega_\delta\),
\[
\begin{aligned}
&-\Delta_{\mathbb S^N}U_\varepsilon(x)
-\frac{\lambda^*}{\sin^2 r}U_\varepsilon(x)
-(\sin r)^\alpha U_\varepsilon(x)^p\\
&=
\varepsilon\left(
-\Delta_{\mathbb S^N}G_\gamma(x)
-\frac{\lambda^*}{\sin^2 r}G_\gamma(x)
\right)
-\varepsilon^p(\sin r)^\alpha G_\gamma(x)^p\\
&\ge
\frac{\varepsilon}{2}
\left(
-\Delta_{\mathbb S^N}G_\gamma(x)
-\frac{\lambda^*}{\sin^2 r}G_\gamma(x)
\right)>0.
\end{aligned}
\]
For \(x\in\Omega_\delta\), set
\[
F_\varepsilon(x)
=
-\Delta_{\mathbb S^N}U_\varepsilon(x)
-\frac{\lambda^*}{\sin^2 r}U_\varepsilon(x)
-(\sin r)^\alpha U_\varepsilon(x)^p.
\]
Then
\[
F_\varepsilon>0
\qquad \text{in } \Omega_\delta.
\]
We choose
\[
f_\varepsilon(x)=\frac12F_\varepsilon(x),
\qquad x\in\Omega_\delta.
\]
Thus,
\[
f_\varepsilon\in C(\overline{\Omega_\delta}^{\,\pi}),
\qquad
f_\varepsilon>0 \quad \text{in } \Omega_\delta,
\]
and
\[
-\Delta_{\mathbb S^N}U_\varepsilon(x)
-\frac{\lambda^*}{\sin^2 r}U_\varepsilon(x)
\ge
(\sin r)^\alpha U_\varepsilon(x)^p+f_\varepsilon(x)
\]
for every \(x\in\Omega_\delta\).

Finally, we use the same time factor as in Case 1. Namely, let
\[
a(t)=1-\frac12 e^{-t},\qquad t>0,
\]
and define
\[
u(t,x)=a(t)U_\varepsilon(x),
\qquad (t,x)\in Q.
\]
Then
\[
u\in C^\infty((0,\infty)\times\Omega_\delta)\cap C(Q),
\qquad
u\ge0 \quad \text{on } (0,\infty)\times\Gamma_\delta,
\]
and
\[
\partial_tu-\Delta_{\mathbb S^N}u-\frac{\lambda^*}{\sin^2 r}u
\ge
(\sin r)^\alpha u^p+f_\varepsilon
\qquad \text{in } (0,\infty)\times\Omega_\delta.
\]
Hence, \(u\) is a classical solution of \eqref{P} corresponding to the positive source 
\(f_\varepsilon\in C(\overline{\Omega_\delta}^{\,\pi})\).

This completes the proof of part (ii), and hence the proof of Theorem~\ref{T1.2}.
\end{proof}

\section{Further remarks and open questions}\label{sec5}

We conclude the paper with a few remarks and open questions concerning the scope of Theorem~\ref{T1.2}.

\section{Further remarks and open questions}\label{sec5}

We conclude the paper with a few remarks and open questions concerning the scope of Theorem~\ref{T1.2}.

First, throughout the proof of Theorem~\ref{T1.2}, the parameter \(\delta\) is assumed to be sufficiently close to \(\pi\). This restriction appears naturally in both the nonexistence and existence parts of the proof. In the nonexistence argument, it is used to construct a positive radial Hardy barrier \(H\), obtained from a singular solution near the antipodal point \(r=\pi\). The positivity of this barrier is guaranteed only in a left-neighborhood of \(\pi\), and this is essential for the construction of admissible test functions.

The same type of restriction also appears in the existence part. There, the explicit stationary supersolutions are constructed from singular radial profiles \(G_q\) in the case \(0<\lambda<\lambda^*\), and \(G_\gamma\) in the case \(\lambda=\lambda^*\). The required inequality
\[
-\Delta_{\mathbb S^N}G-\frac{\lambda}{\sin^2 r}G>0
\]
is obtained from the asymptotic behavior of these profiles near \(r=\pi\). Hence, it is guaranteed only in a sufficiently small neighborhood of the antipodal singular point. This forces \(\delta\) to be chosen sufficiently close to \(\pi\) in the construction of solutions as well.

At this stage, it is not clear whether the restriction \(\delta\) close to \(\pi\) is merely technical or whether it reflects a genuine feature of the problem. On the one hand, the critical mechanism identified here is localized near the singular point \(r=\pi\), where the potential \(\lambda/\sin^2 r\) dominates. On the other hand, when the domain extends farther away from the singular point, the global behavior of the associated radial Hardy equation may influence the existence and nonexistence theory. It would therefore be interesting to determine whether the same critical picture remains valid for arbitrary \(\delta\in(0,\pi)\), or whether additional phenomena may occur.

A second open question concerns the critical Hardy case. For \(\alpha>-2\), the critical exponent obtained in this paper is
\[
p_{\mathrm{crit}}
=
1+\frac{\alpha+2}{\lambda_N}.
\]
When \(0<\lambda<\lambda^*\), our nonexistence result covers the critical case
\[
p=p_{\mathrm{crit}}.
\]
When \(\lambda=\lambda^*\), however, the critical nonexistence result is proved only under the additional restriction
\[
-2<\alpha<\frac{N-6}{2}.
\]
Thus, the case
\[
\lambda=\lambda^*,
\qquad
\alpha\ge \frac{N-6}{2},
\qquad
p=p_{\mathrm{crit}},
\]
is not covered by Theorem~\ref{T1.2}. The obstruction comes from the logarithmic behavior of the Hardy barrier in the critical case. More precisely, the logarithmic gain produced by the cutoff is compensated by a logarithmic loss in the estimates, and our method no longer yields a vanishing upper bound in this range.

It remains open whether nonexistence still holds in this remaining critical regime, or whether one can construct positive solutions for some positive sources. Resolving this question would require either sharper logarithmic test-function estimates or a different construction adapted to the double-root structure of the critical Hardy operator.

These questions indicate that the present results give a sharp critical picture within the range accessible by the radial Hardy-barrier method, while leaving open the possible influence of global geometry and critical logarithmic effects.

\section*{Funding}

The first author is supported by Ongoing Research Funding Program, (ORF-2026-57), King Saud University, Riyadh, Saudi Arabia.

\section*{Declaration of competing interest}
The authors declare that there is no conflict of interest.

\section*{Data Availability Statements}
The manuscript has no associated data

\end{document}